\documentclass[12pt]{article}%
\usepackage{amsfonts}
\usepackage{mitpress}%
\usepackage{amsmath}%
\usepackage{amssymb}%
\usepackage{graphicx}
\providecommand{\U}[1]{\protect\rule{.1in}{.1in}}
\newtheorem{theorem}{Theorem}

\newtheorem{claim}[theorem]{Claim}

\newtheorem{conjecture}[theorem]{Conjecture}

\newtheorem{lemma}[theorem]{Lemma}

\newtheorem{proposition}[theorem]{Proposition}

\newenvironment{proof}[1][Proof]{\noindent\textbf{#1.} }{\ \rule{0.5em}{0.5em}}
\newdimen\dummy
\dummy=\oddsidemargin
\begin{document}

\title{Inconsistency Probability of Sparse Equations over $\mathbb{F}_{2}$}
\author{Peter Horak$^{1}$
$\vert$
Igor Semaev$^{2}$\\$^{1}$University of Washington, USA
$\vert$
$^{2}$University of Bergen, Norway\\Correspondence: Peter Horak (horak@uw.edu)}
\maketitle

\begin{abstract}
Let $n$ denote the number of variables and $m$ the number of equations in a
sparse polynomial system over the binary field. We study the inconsistency
probability of randomly generated sparse polynomial systems over the binary
field, where each equation depends on at most $k$ variables and the number of
variables grows. Associating the system with a hypergraph, we show that the
inconsistency probability depends strongly on structural properties of this
hypergraph, not only on $n,m,$ and $k.$ Using inclusion--exclusion, we derive
general bounds and obtain tight asymptotics for complete $k$-uniform
hypergraphs. In the 2-sparse case, we provide explicit formulas for paths and
stars, characterize extremal trees and forests, and conjecture a formula for
cycles. These results have implications for SAT solving and cryptanalysis.

\end{abstract}

\section{Introduction}

Let $n,m,k$ be natural numbers, and $X$ denote a set of $n$ variables taking
values in the binary field $\mathbb{F}_{2}$. Also, let $X_{i},1\leq i\leq m,$
be non-emtpy subsets of $X$ and let $f_{i}(X_{i}),i=1,\ldots,m,$ be
polynomials over $\mathbb{F}_{2}$ that depend only on the variables in $X_{i}
$.

\subsection{Problem}

In this work, we aim to analyze the solutions in $\mathbb{F}_{2}^{n}${} to the
system of equations
\begin{equation}
f_{1}(X_{1})=0,\,f_{2}(X_{2})=0,\ldots,f_{m}(X_{m})=0.\label{system}%
\end{equation}
Such system is called $k$-\emph{sparse} if $|X_{i}|\leq k$, where $k$ is a
fixed small number (such as $3$ or $4$), while $n$ is large. As usual,
(\ref{system}) is called \emph{consistent} (solvable) if it has a solution in
$\mathbb{F}_{2}^{n},$ otherwise it will be called \emph{inconsistent}.

We say that the system \eqref{system} is \emph{randomly generated} if the
variable subsets $X_{i}$ are fixed (allowing for possible repetitions), and
the polynomials $f_{i}$ are chosen uniformly and independently at random. A
hypergraph $\mathcal{X}=(X,\{X_{1},X_{2},\ldots,X_{m}\}),$ where the vertex
set is $X$ and the edges are subsets $X_{i}$ may be associated with
(\ref{system}). In this paper, we study inconsistency probability
$p(\mathcal{X)}$ for randomly generated systems (\ref{system}). In particular,
we demonstrate that $p(\mathcal{X)}$ depends significantly on the structural
properties of the underlying hypergraph $\mathcal{X}$, and not solely on the
parameters $n,m,$ and $k$.\ 

\subsection{Motivation}

It was observed in \cite{iS2013} that modern SAT solvers such as MiniSat
\cite{MiniSat} exhibit sub-exponential behavior - rather than exponential -
when applied to randomly generated sparse equation systems \eqref{system}, as
the number of variables increases. This is a surprising result, given that
solving such systems (i.e., determining their satisfiability) is NP-complete,
as the problem is reducible to a SAT problem.

In this work, we take initial steps toward understanding and explaining this
phenomenon. In Section \ref{main_results}, we demonstrate that the
inconsistency probability $p(\mathcal{X)}$ is significantly higher in the
sparse case. This likely explains the surprising efficiency of SAT solvers in
such instances. Once a small number of variables are fixed to constants, the
solver often detects a contradiction in a resulting inconsistent subsystem.
When a contradiction is found, the current assignment is rejected and a new
one is tried---effectively navigating a search tree to solve the system.

Building on this idea, we show that a lower bound on the inconsistency
probability for certain randomly generated systems can lead to an upper bound
on the expected complexity of solving \eqref{system}. Suppose the system
(\ref{system}) is divided into two subsystems:
\begin{equation}
f_{1}(X_{1})=0,\ldots,f_{k}(X_{k})=0,\label{part1}%
\end{equation}
and
\begin{equation}
f_{k+1}(X_{k+1})=0,\ldots,f_{m}(X_{m})=0,\label{part2}%
\end{equation}
where $Y_{1}=X_{1}\cup\ldots\cup X_{k}$ and $Y_{2}=X_{k+1}\cup\ldots\cup
X_{m}$, and $Y\subseteq Y_{1}\cap Y_{2}$.

Let $q_{1}$ and $q_{2}$ denote the probabilities that (\ref{part1}) and
(\ref{part2}), respectively, are consistent for a given assignment $Y=a$.
These probabilities do not depend on $a$ since the equations are randomly
generated. Let $Q_{1}$ and $Q_{2}${} denote the average complexities of
solving (\ref{part1}) and (\ref{part2}) after this fixation. Then, the average
complexity of solving (\ref{system}) is:%
\begin{equation}
2^{|Y|}\left(  Q_{1}+q_{1}Q_{2}+q_{1}q_{2}Q\right)  ,\label{complexity}%
\end{equation}
where $Q$ represents the complexity of combining the solutions of
(\ref{part1}) and (\ref{part2}) that agree on the variables in $Y$, into a
solution of (\ref{system}).

To justify this, note:

\begin{itemize}
\item $2^{|Y|}Q_{1}$ is the expected complexity of solving (\ref{part1})
across all fixations of $Y$.

\item The expected number of fixations for which (\ref{part1}) is consistent
is $2^{|Y|}q_{1}.$

\item For each such fixation, (\ref{part2}) must be solved, contributing
$2^{|Y|}q_{1}Q_{2}$ to the total complexity.

\item The number of fixations $Y=a$ where both parts are consistent is
$2^{|Y|}q_{1}q_{2}Q${}, and combining their solutions adds the final term.
\end{itemize}

This leads to the total complexity bound stated in \eqref{complexity}.

Clearly, the system can be split into more than two subsystems, potentially
further reducing complexity. In practice, this can be done dynamically by
attempting to identify inconsistent subsystems under each variable fixation.
Hence, upper bounds on the consistency probabilities $q_{1},q_{2}${}, or lower
bounds on the corresponding inconsistency probabilities, imply upper bounds on
the average complexity of solving (\ref{system}).

\subsection{Sparse equations in cryptanalysis}

The definition of the sparse equation system \eqref{system} has been motivated
by applications in cryptanalysis, where the goal is to invert a discrete
function in order to break a cipher. The mappings implemented by ciphers
typically consist of compositions of several functions, each involving only a
small number of variables---such as S-boxes, permutations, and linear layers
in modern block ciphers. These mappings can be represented using a small
number of simple gates, as they are designed to be efficiently computable. To
simplify the equations describing the cipher and to obtain a sparse system of
type \eqref{system}, intermediate variables may be introduced. Although the
resulting system may involve a large number of variables, it usually remains
manageable in terms of memory usage and computational manipulation. The
interested reader is directed to \cite{Bard} and\ \cite{Odlyzko}.

\subsection{Sparse equations and $k$-SAT problem}

The probability of consistency (satisfiability) has been extensively studied
for $k$-SAT problems in relation to the so-called phase transition phenomenon.
The Boolean $k$-satisfiability ($k$-SAT) problem involves determining whether
a CNF (conjunctive normal form) formula - composed of a conjunction of clauses
of length $k$ - is satisfiable. Each clause is a disjunction of $k$ literals,
where a literal is either a variable $x\in X$ or its negation $\bar{x}$. We
assume that no clause contains both a variable and its negation
simultaneously. The problem of solving the system of $k$-sparse equations
\eqref{system} can be represented as an instance of the $k$-SAT problem, and
vice versa.

For a variable set $|X|=n$, the number of possible clauses of size $k$ is
${\binom{n}{k}}2^{k}$. Let $m$ be a natural number. A random instance of
$k$-SAT is one in which the CNF consists of $m$ clauses chosen independently
and uniformly at random---with replacement---from the set of all possible
$k$-clauses. (Other notions of randomness have also been considered.)

It is widely believed that random $k$-SAT instances exhibit a phase
transition. Specifically, it was conjectured in \cite{CR} that for every
$k\geq2$, there exists a threshold $r_{k}$ such that for any $\varepsilon>0$ ,
the following holds as $n\rightarrow\infty:$ if $m/n<r_{k}-\varepsilon$, then
the CNF is satisfiable with high probability; if $m/n>r_{k}+\varepsilon$ the
CNF is unsatisfiable with high probability.

This conjecture was proved for $k=2$ with the threshold $r_{2}=1$
\cite{Goerdt92}. For $k=3,$ it has been shown that if $m/n<3.145$, a random
CNF is almost always satisfiable, and if $m/n>4.596$, it is almost always
unsatisfiable, \cite{Ach00} and \cite{JSV00}. The threshold for $3$-SAT is
believed to be around $4.2,$ \cite{GS}. For sufficiently large $k$, the
conjecture was proved in \cite{DSS}.

\section{Bounds on inconsistency probability}

\label{main_results}

The main results of this section provide bounds on inconsistency probability
$p(\mathcal{X})$ for both general hypergraph $\mathcal{X}$ and the complete
$k$-uniform hypergraph $\mathcal{X}_{k}=(X,\{X_{1},X_{2},\ldots,X_{m}\})$,
where $X_{i}$ are all different $k$-subsets of $X;$ i.e., $m={\binom{n}{{k}}}
$.\ We begin with two propositions that will be used frequently in the proofs
throughout this section.

\subsection{Auxiliary results}

\begin{proposition}
\label{AA}Let (\ref{system}) be a randomly generated system with $m$
equations, and $\left\vert X\right\vert =n$. Then

\begin{enumerate}
\item each equation $f_{i}\mathcal{(}X_{i})=0$ taken separately is
inconsistent with probability $\frac{1}{2^{2^{|X_{i}|}}},$

\item any $a\in\mathbb{F}_{2}^{n}$ is a solution to the equation
$f_{i}\mathcal{(}X_{i})=0$ with probability $1/2,$

\item any $a\in\mathbb{F}_{2}^{n}$ is a solution of entire system
(\ref{system}) with probability $\frac{1}{2^{m}.}$
\end{enumerate}
\end{proposition}

\begin{proof}
Let $\left\vert X_{i}\right\vert =k.$ There is a natural one-to-one
correspondence between the set of all $2^{2^{k}}$ Boolean polynomials over
$\mathbb{F}_{2}$ in $k$ variables, and the set of all subsets of
$\mathbb{F}_{2}^{k},$ where each polynomial corresponds to its set of roots.
Exactly one polynomial, namely $f(X_{i})=1$, corresponds to the empty set.
Hence, the probability that a randomly chosen polynomial $f_{i}\mathcal{(}%
X_{i})$ results in an inconsistent equation $f_{i}\mathcal{(}X_{i})=0$ is
$\frac{1}{2^{2^{k}}}.$ Since any $a\in\mathbb{F}_{2}^{k}$ lies in \ exactly a
half of all subsets of $\mathbb{F}_{2}^{k},$ the probability that $a$ is a
root of a randomly chosen polynomial is $\frac{1}{2}.$ As the equations in the
system (\ref{system}) are independently generated, the probability that $a$ is
a solution of all $m$ equations is $\frac{1}{2^{m}}.$
\end{proof}

The following proposition is almost trivial and was already used for
estimating the complexity of solving (\ref{system}) in \cite{iS09, iS2013}.

\begin{proposition}
\label{BB} Let $\mathcal{X}=(X,\{X_{1},X_{2},\ldots,X_{m}\})$ and let
(\ref{system}) be a randomly generated equation system. Then

\begin{enumerate}
\item $p(\mathcal{X})\geq1-\prod_{i=1}^{m}(1-2^{-2^{|X_{i}|}})$.

\item If $X_{1},\ldots,X_{m}$ are pairwise disjoint, then $p(\mathcal{X}%
)=1-\prod_{i=1}^{m}(1-2^{-2^{|X_{i}|}}).$
\end{enumerate}
\end{proposition}

\begin{proof}
If the system (\ref{system}) is consistent then each of the equations
$f_{i}(X_{i})=0$ is consistent. By Proposition \ref{AA} the probability that
$f_{i}(X_{i})=0$ has a solution (i.e., is consistent) is $1-2^{-2^{|X_{i}|}}$.
So, $1-p(\mathcal{X})\leq\prod_{i=1}^{m}(1-2^{-2^{|X_{i}|}})$. That implies
the first statement.

If the variable sets $X_{1},\ldots,X_{m}$ are pairwise disjoint, then the
events are independent, which implies the second statement. Clearly, in this
case, system (\ref{system}) has the maximum probability of consistency.
\end{proof}

We illustrate how the size of the variable subsets influences the convergence
of the inconsistency probability $p(\mathcal{X}).$

Let $\mathcal{X}^{\prime}=(X,\{X_{1},X_{2},\ldots,X_{m}\}),$ where $|X_{i}|=k$
and let $\mathcal{X}^{\prime\prime}=(X,\{X_{1},X_{2},\ldots,X_{m}\}),$ where
$X_{i}=X$ for all $i.$ Proposition \ref{BB} implies that $p(\mathcal{X}%
^{\prime})\geq1-\left(  1-2^{-2^{k}}\right)  ^{m}$ and therefore
$p(\mathcal{X}^{\prime})\rightarrow1$ as $m$ grows and $k$ is fixed.

As for $\mathcal{X}^{\prime\prime}$, by Proposition \ref{AA}, the probability
that a fixed assignment $X=a$ is a solution to (\ref{system}) is $1/2^{m}$.
Hence, $a$ is not a solution with probability $1-1/2^{m}$. Overall, since $a$
can take $2^{n}$ possible values, we have $p(\mathcal{X}^{\prime\prime
})=\left(  1-\frac{1}{2^{m}}\right)  ^{2^{n}}$.

Therefore the probability $p(\mathcal{X}^{\prime\prime})$ is close to $1$ only
if $m$ is significantly larger than $n$. In contrast, $p(\mathcal{X}^{\prime
})\rightarrow1$ whenever $m$ increases. We note that in estimating
$p(\mathcal{X}^{\prime})$ the structural properties of the hypergraph
$\mathcal{X}^{\prime}$ have not been considered.

\subsection{Inconsistency probability for systems with $2$ and $3$ equations}

Proposition \ref{m=2} provides a convenient formula for $p(\mathcal{X})$ when
$m=2$. In Proposition \ref{m=3}, we show that the case $m=3$ reduces to
computing the probability that a random tripartite graph contains no $3$-cycle.

\begin{proposition}
\label{m=2} Let $\mathcal{X}=(X,\{X_{1},X_{2}\})$ and let \eqref{system} be a
randomly generated equation system. Then
\[
p(\mathcal{X})=\left(  1-(1-\frac{1}{2^{2^{|X_{1}\setminus X_{2}|}}}%
)(1-\frac{1}{2^{2^{|X_{2}\setminus X_{1}|}}})\right)  ^{2^{|X_{1}\cap X_{2}|}%
}.
\]

\end{proposition}

\begin{proof}
\textbf{\ }For every fixed assignment of the variables in $X_{1}\cap X_{2}${},
the remaining variables involved in the equations $f_{1}(X_{1})=0$ and
$f_{2}(X_{2})=0$ become disjoint. Under such a fixed assignment, the events
that $f_{1}$ and $f_{2}$ are satisfied become independent. Therefore, the
probability that the system is inconsistent (i.e., that at least one of the
equations is not satisfied) is $1-(1-\frac{1}{2^{2^{|X_{1}\setminus X_{2}|}}%
})(1-\frac{1}{2^{2^{|X_{2}\setminus X_{1}|}}})$. Since the assignments to
$X_{1}\cap X_{2}$ are independent, this implies the claim.
\end{proof}

Let $\mathcal{X}=(X,\{X_{1},X_{2},X_{3}\}),$ and let
\begin{equation}
f_{1}(X_{1})=0,f_{2}(X_{2})=0,f_{3}(X_{3})=0\label{system_m=3}%
\end{equation}
be a system of $m=3$ equations. Denote $\{i,j,k\}=\{1,2,3\}$. For every
fixation of variables $X_{i}\cap X_{j}\cap X_{k}$ by constants, we will define
a tripartite graph $G$ on vertex sets $Y_{ij},\,\{i,j\}\subseteq\{1,2,3\}$
defined below. Let
\[
X_{ij}=(X_{i}\cap X_{j})\setminus(X_{i}\cap X_{j}\cap X_{k}).
\]
The vertices $Y_{ij},$ are all binary $|X_{ij}|$-strings, thus the values of
$X_{ij}$. If $X_{ij}=\emptyset$, then $Y_{ij}$ consists of only one vertex.
So, $|Y_{ij}|=2^{|X_{ij}|}$ anyway. Assume a solution to \eqref{system_m=3}.
Denote the projection of that solution to the variables $X_{12}\cup X_{13}$ by
$(a,b)$. Similarly, let the projection of the solution to the variables
$X_{13}\cup X_{23}$ be $(b,c)$, and to the variables $X_{23}\cup X_{12}$ be
$(c,a)$. Thus, $(a,b)$ is an edge in the graph $G$ between $Y_{12}$ and
$Y_{13}$, and $(b,c)$ is an edge between $Y_{12}$ and $Y_{23}$, and $(c,a)$ is
an edge between $Y_{13}$ and $Y_{23}$. Those edges constitute a cycle of
length $3$ in the graph. Obviously, the system \eqref{system_m=3} admits a
solution if and only if the graph $G$ has a cycle of length $3$.

When the system \eqref{system_m=3} is randomly generated, then $G$ is a
randomly generated tripartite graph, where the probability of an edge between
$Y_{ik}$ and $Y_{kj}$ is $q_{k}=1-2^{-2^{|X_{k}\setminus(X_{i}\cup X_{j})|}}$.
Let $q$ denote the probability that a graph $G$ generated this way does not
have a cycle of length $3$. Therefore, $q$ is the probability that the system
\eqref{system_m=3} is inconsistent after the fixation of variables $X_{i}\cap
X_{j}\cap X_{k}$. The probability $q$ does not depend on the fixation itself.
Since the assignments to $X_{1}\cap X_{2}\cap X_{3}$ are independent, we have

\begin{proposition}
\label{m=3}Let $\mathcal{X}=(X,\{X_{1},X_{2},X_{3}\})$ and \eqref{system} be a
randomly generated equation system. Then%
\[
p(\mathcal{X})=q^{2^{|X_{1}\cap X_{2}\cap X_{3}|}}
\]

\end{proposition}

Estimating $q$ seems a difficult problem.

\subsection{Main Results}

By Proposition \ref{BB}, for every hypergraph $\mathcal{X}=(X,\{X_{1}%
,X_{2},\ldots,X_{m}\})$ it holds that $p(\mathcal{X})\geq1-\prod_{i=1}%
^{m}(1-2^{-2^{|X_{i}|}})$. Therefore, $p(\mathcal{X})\rightarrow1$ as $m$
grows. This bound treats the sets independently and does not reflect how they
intersect or overlap, which can lead to suboptimal estimates when the
intersection structure is nontrivial. In general, $p(\mathcal{X})$ tends to
$1$ much more rapidly. In this subsection, we prove this fact for a general
hypergraph $\mathcal{X}=(X,\{X_{1},X_{2},\ldots,X_{m}\}),$ as well as for the
complete $k$-uniform hypergraph $\mathcal{X}_{k}$. \ 

The proof of both bounds is based on the inclusion-exclusion principle. Let
$S$ denote the set of all the solutions to \eqref{system}, and let
$q(\mathcal{X})$ be the probability of at least one solution. Thus,
$q(\mathcal{X})=1-p(\mathcal{X})=\mathbf{Pr}(|S|>0)$. In our setting, the
inclusion-exclusion principle can be expressed as follows:

\begin{theorem}
\label{inclusion_exclusion}
\[
q(\mathcal{X})=\sum_{a}\mathbf{Pr}(a\in S)-\sum_{a,b}\mathbf{Pr}(a,b\in
S)+\sum_{a,b,c}\mathbf{Pr}(a,b,c\in S)-\ldots,
\]
where the first sum is over all $n$-bit strings $a$, the second sum is over
all pairs of different $n$-bit strings $a,b$, the third sum is over all
triplets of different $n$-bit strings $a,b,c$, etc. In particular,
\begin{equation}
\sum_{a}\mathbf{Pr}(a\in S)-\sum_{a,b}\mathbf{Pr}(a,b\in S)\leq q(\mathcal{X}%
)\leq\sum_{a}\mathbf{Pr}(a\in S).\label{inc}%
\end{equation}

\end{theorem}

\begin{theorem}
\label{inclusion_exclusion_1}For every hypergraph $\mathcal{X}=(X,\{X_{1}%
,X_{2},\ldots,X_{m}\}),$ it holds that
\[
\frac{3\cdot2^{n}}{2^{m+1}}-\frac{2^{2n}}{2^{2m+1}}\sum_{\{i_{1},..,i_{r}%
\}}2^{-|X_{i_{1}}\cup..\cup X_{i_{r}}|}\leq q(\mathcal{X})\leq\frac{2^{n}%
}{2^{m}}.
\]
where the sum runs over all subsets $\{i_{1},..,i_{r}\}\subseteq\{1,..,m\}$.
\end{theorem}

\textbf{Remark 1. }The theorem implies that $p(\mathcal{X)}\geq1-\frac{2^{n}%
}{2^{m}},$ while Proposition \ref{BB} provides the bound $p(\mathcal{X}%
)\geq1-\prod_{i=1}^{m}(1-2^{-2^{|X_{i}|}}\dot{)}.$ For $m>\frac{n}{\log
_{2}(3/2)} $ we have
\[
2^{n}<\prod_{i=1}^{m}2(1-2^{-2})<\prod_{i=1}^{m}2(1-2^{-2^{|X_{i}|}}).
\]
So, $1-\frac{2^{n}}{2^{m}}>1-\prod_{i=1}^{m}(1-2^{-2^{|X_{i}|}}\dot{)}$ and
$p(\mathcal{X})\rightarrow1$ at a faster rate than one might deduce from
Proposition \ref{BB}.

\textbf{Remark 2.} In the general case, it is hard to estimate the sum
$\sum_{\{i_{1},..,i_{r}\}}2^{-|X_{i_{1}}\cup..\cup X_{i_{r}}|}$ in terms of
$n$ and $m$, and therefore it is difficult to assess how close the lower and
upper bounds on $q(\mathcal{X})$ are.

\begin{proof}
\textbf{of Theorem \ref{inclusion_exclusion_1}. } By Proposition \ref{AA},
$\mathbf{Pr}(a\in S)=\frac{1}{2^{m}}$ and therefore $\sum_{a}\mathbf{Pr}(a\in
S)=\frac{2^{n}}{2^{m}}$, where the sum is over all $n$-bit strings $a$.

We now find an expression for $\sum_{[a,b]}\mathbf{Pr}(a,b\in S)$, where
$[a,b]$ denotes the ordered pair $a,b$ including $a=b$. The value of
$\sum_{a,b}\mathbf{Pr}(a,b\in S)$ follows from the identity $\sum
_{[a,b]}\mathbf{Pr}(a,b\in S)=2\sum_{a,b}\mathbf{Pr}(a,b\in S)+\sum
_{a}\mathbf{Pr}(a\in S)$, so the theorem follows from the lemma below.

\begin{lemma}
\label{inclusion_exclusion_second_term}
\[
\sum_{[a,b]}\mathbf{Pr}(a,b\in S)=\frac{2^{2n}}{2^{2m}}\sum_{\{i_{1}%
,..,i_{r}\}}2^{-|X_{i_{1}}\cup..\cup X_{i_{r}}|},
\]
where the sum is over all subsets $\{i_{1},..,i_{r}\}\subseteq\{1,..,m\} $.
\end{lemma}

Let $T$ be a set of variables, a subset of $X$ and let $m_{T}$ denote the
number of $X_{i}\subseteq T$. Then
\begin{align*}
\sum_{\lbrack a,b]}\mathbf{Pr}(a,b  & \in S)=\sum_{T}\sum_{a=b(T)}%
\mathbf{Pr}(a,b\in S)\\
& =\sum_{T}\frac{2^{n}}{2^{m_{T}}\,4^{m-m_{T}}}=\frac{2^{n}}{4^{m}}\sum
_{T}2^{m_{T}},
\end{align*}
where $a=b(T)$ means those $a,b$ which coincide exactly on $T$. The second
equality is true because $\mathbf{Pr}(a,b\in S)=\frac{1}{2^{m_{T}}%
\,4^{m-m_{T}}}$ if $a=b(T)$. Now
\[
\sum_{T}2^{m_{T}}=\sum_{s=0}^{m}2^{s}\,\sum_{\{j_{1},..,j_{s}\}}%
C_{j_{1},..,j_{s}},
\]
where $C_{j_{1},..,j_{s}}$ is the number of subsets $T$ which contain exactly
$s$ variable sets $X_{j_{1}},..,X_{j_{s}}$. By inclusion-exclusion principle,
\[
C_{j_{1},..,j_{s}}=\sum_{\{i_{1},..,i_{t}\}}(-1)^{t}2^{n-|X_{j_{1}}\cup..\cup
X_{j_{s}}\cup X_{i_{1}}\cup..\cup X_{i_{t}}|},
\]
where the sum is over all subsets $\{i_{1},..,i_{t}\}\subseteq
\{1,..,m\}\setminus\{j_{1},..,j_{s}\}$. We denote $\{i_{1},..,i_{t}%
,j_{1},..,j_{s}\}=\{i_{1},..,i_{r}\}$ and get
\[
\sum_{T}2^{m_{T}}=\sum_{\{i_{1},..,i_{r}\}}2^{n-|X_{i_{1}}\cup..\cup X_{i_{r}%
}|}\sum_{s=0}^{r}2^{s}{\binom{r}{{s}}}(-1)^{r-s}=\sum_{\{i_{1},..,i_{r}%
\}}2^{n-|X_{i_{1}}\cup..\cup X_{i_{r}}|}.
\]
Therefore,
\[
\sum_{\lbrack a,b]}\mathbf{Pr}(a,b\in S)=\frac{2^{2n}}{2^{2m}}\sum
_{\{i_{1},..,i_{r}\}}2^{-|X_{i_{1}}\cup..\cup X_{i_{r}}|}.
\]
where the sum is over all subsets $\{i_{1},..,i_{r}\}\subseteq\{1,..,m\}$.
\ The lemma is proved.

We have
\begin{align}
& \sum_{a}\mathbf{Pr}(a\in S)-\sum_{a,b}\mathbf{Pr}(a,b\in S)\nonumber\\
& =\sum_{a}\mathbf{Pr}(a\in S)-\frac{1}{2}(\sum_{[a,b]}\mathbf{Pr}(a,b\in
S)-\sum_{a}\mathbf{Pr}(a\in S))\nonumber\\
& =\frac{3}{2}\sum_{a}\mathbf{Pr}(a\in S)-\frac{1}{2}\sum_{[a,b]}%
\mathbf{Pr}(a,b\in S).\nonumber
\end{align}
Combining \eqref{inc} and Lemma \ref{inclusion_exclusion_second_term} yields
\[
\frac{3\cdot2^{n}}{2^{m+1}}-\frac{2^{2n}}{2^{2m+1}}\sum_{i_{1},..,i_{r}%
}2^{-|X_{i_{1}}\cup..\cup X_{i_{r}}|}\leq q(\mathcal{X})\leq\frac{2^{n}}%
{2^{m}}.
\]
That completes the proof of Theorem \ref{inclusion_exclusion_1}.
\end{proof}

\section{Inconsistency Probability for a Complete $k$-hypergraph}

Let $\mathcal{X}_{k}=(X,\{X_{1},X_{2},\ldots,X_{m}\})$ be the complete
$k$-uniform hypergraph, where $m={\binom{n}{{k}}}$ and $X_{i}$ are all
different $k$-subsets of $X$. In this case, a tight asymptotic bound holds.

\begin{theorem}
\label{X_0} Let $k\geq2$. Then $p(\mathcal{X}_{k})=1-\frac{2^{n}}%
{2^{{\binom{n}{k}}}}+O\left(  \frac{n2^{n}}{2^{{\binom{n}{k}}+{\binom
{n-1}{k-1}}}}\right)  $.
\end{theorem}

\begin{proof}
The proof of this theorem is also based on the inequality (\ref{inc}). The
total number of edges (i.e. equations) in $\mathcal{X}_{k}$ is ${\binom{n}%
{{k}}.}$ By Proposition \ref{AA}, we have $\mathbf{Pr}(a\in S)=\frac
{1}{2^{{\binom{n}{k}}}}$ and so
\begin{equation}
\sum_{a}\mathbf{Pr}(a\in S)=\frac{2^{n}}{2^{{\binom{n}{k}}}}.\label{dd}%
\end{equation}
In order to compute $\sum_{a,b}\mathbf{Pr}(a,b\in S)$ we partition the set of
all unordered pairs of distinct $n$-bit strings $a,b$ into $n$ classes. A pair
$\{a,b\}$ belongs to the class $C_{t}$ if $a$ and $b$ agree in exactly $0\leq
t\leq n-1$ positions. The size of each class is $|C_{t}|={\binom{n}{{t}}%
}2^{n-1}.$

Let $\{a,b\}\in C_{t}$. \ Then the pair $\{a,b\}$ has ${\binom{t}{{k}}}$
projections of size $1$ and ${\binom{n}{{k}}}-{\binom{t}{{k}}}$ projections of
size $2$ onto the variable sets $X_{i}$. Therefore,
\[
\mathbf{Pr}(a,b\in S)=\frac{1}{2^{{\binom{t}{{k}}}+2({\binom{n}{{k}}}%
-{\binom{t}{{k}}})}}=\frac{1}{2^{-{\binom{t}{{k}}}+2{\binom{n}{{k}}}}}.
\]
Then
\begin{align}
\sum_{a,b}\mathbf{Pr}(a,b\in S)  & =\sum_{t=0}^{n-1}\,\sum_{\{a,b\}\in C_{t}%
}\frac{1}{2^{-{\binom{t}{{k}}}+2{\binom{n}{{k}}}}}\nonumber\\
& =\sum_{t=0}^{n-1}\,\frac{{\binom{n}{{t}}}\,2^{n-1}}{2^{-{\binom{t}{{k}}%
}+2{\binom{n}{{k}}}}}=\frac{2^{n}}{2^{1+2{\binom{n}{{k}}}}}\sum_{t=0}%
^{n-1}{\binom{n}{{t}}}\,2^{{\binom{t}{{k}}}}\nonumber\\
& =O\left(  \frac{n2^{n}}{2^{{\binom{n}{{k}}}+{\binom{n-1}{{k-1}}}}}\right)
.\nonumber
\end{align}
The last estimate follows from $\sum_{t=0}^{n-1}{\binom{n}{{t}}}%
\,2^{{\binom{t}{{k}}}}=O(n2^{\binom{n-1}{{k}}})$ as $k\geq2$ is fixed and $n$
tends to infinity. To complete the proof of the theorem it now suffices to
substitute the bound on $\sum_{a,b}\mathbf{Pr}(a,b\in S)$ and the value of
$\sum_{a}\mathbf{Pr}(a\in S)$ provided in (\ref{dd}) into the expression in
(\ref{inc}).
\end{proof}

\section{\bigskip$2$-sparse system of equations}

In this section we consider a randomly generated $2$-sparse equation system
\eqref{system} over the variable set $X=\{x_{1},...,x_{n}\}$. In what follows,
the system will be denoted by $\mathcal{S}$. The associated hypergraph is
$2$-uniform and hence corresponds to a simple graph $G=G(V,E)$ with $n$
vertices $V$ and $m$ edges $E$. The vertices $V=\{v_{1},...,v_{n}\}$
correspond to the variables $X$, and each edge $v_{i}v_{j}$ represents an
equation involving the live variables $x_{i}$ and $x_{j}.$

In this section we will study the probability $q(G)$ that a randomly generated
system is consistent for some particular graphs $G$.

There are $2^{4}$ polynomial functions in two variables over $GF(2)$, so the
graph $G$ can represent $2^{4m}$ distinct systems of equations. Thus,
$q(G)=\frac{N}{2^{4m}},$ where $N$ is the number of solvable systems
corresponding to $G.$

We characterize the trees with extremal consistency probabilities and the
forests that maximize consistency probability. In addition, we provide
explicit formulas for the consistency probabilities of a path and a star.
Finally, we conjecture an explicit formula for the consistency probability of
a cycle and present supporting evidence for the conjecture.

\subsection{Preliminaries}

We begin by introducing some notation and stating preliminary results. Let $G
$ be a graph associated with a system of equations $\mathcal{S}$. For any
vertex $v$ in $G $ we set $v^{0},v^{1},$ and $v^{01}$ to be the probability
that the variable $x_{v}$ equals to $0$ in all solutions, the variable $x_{v}
$ equals to $1$ in all solutions, and there is a solution with $x_{v}=0 $ and
another solution of $\mathcal{S}$ with $x_{v}=1$, respectively. Clearly,
$q(G)=v^{0}+v^{1}+v^{01}.$

\begin{lemma}
\label{L1}Let $G$ be a graph. Then for each vertex $v$ in $G,$
\[
v^{0}=v^{1}
\]
and thus%
\[
q(G)=2v^{0}+v^{01}.
\]

\end{lemma}

\begin{proof}
\textbf{\ }For the sake of simplicity, in what follows we assume wlog that the
vertex $v$ corresponds to the variable $x_{1}.\,\ $To prove that $v^{0}%
=v^{1},$ we consider the mapping $h_{v}:\mathcal{F\rightarrow F}%
,\mathcal{\,\ }$where $\mathcal{F}$ is the set of all polynomials over $GF(2)$
in $n$ variables, such that $h_{v}(f)=f_{v},$ where $(c_{1}+1,c_{2}%
,...,c_{n})$ is a root of $f_{v}$ iff $(c_{1},c_{2},...,c_{n})$ is a root of
$f$ ($c_{1}+1$ is carried out in $GF(2)$). Obviously, $h_{v}$ is well defined.

Now we consider a mapping $H_{v}$ induced by $h_{v}.$ For a solvable system
$\mathcal{S}$ corresponding to $G$ we set $H_{v}(\mathcal{S)=S}_{v}$, where
$\mathcal{S}_{v}=\{f_{v},f\in\mathcal{S\}}$. Clearly, $\mathcal{S}_{v}$ is
solvable iff $\mathcal{S}$ is, and if for all solutions to $\mathcal{S}$ it is
$x_{v}=0$, then for all solutions to $\mathcal{S}_{v}$ we have $x_{v}=1.$ This
establishes that $v^{0}=v^{1}$.
\end{proof}

A well-known result, used repeatedly in what follows, states that

\begin{claim}
\label{Claim}Let $f(x,y)=0$ \ be a randomly selected equation on two
variables. Then $v_{x}^{0}=\frac{3}{16},v_{x}^{01}=\frac{9}{16}.$
\end{claim}

Let $G=G(V,E)$ be a graph such that an edge $vw\notin E$. \ Then by $G+vw$ we
denote a graph that results from $G$ by adding this edge.

\begin{lemma}
\label{L4}Let $v$ be a vertex of a graph $G,$ and $w$ be a vertex not in $G.$
Then,%
\begin{align*}
\text{(a) }w_{G+vw}^{0}  & =\frac{1}{2}v_{G}^{0}+\frac{3}{16}v_{G}^{01},\text{
and }w_{G+vw}^{01}=\frac{1}{2}v_{G}^{0}+\frac{9}{16}v_{G}^{01};\\
\text{(b) }v_{G+vw}^{0}  & =\frac{3}{4}v_{G}^{0}+\frac{3}{16}v_{G}^{01},\text{
and }v_{G+vw}^{01}=\frac{9}{16}v^{01}_{G};
\end{align*}
and%
\[
\text{(c) }q(G+vw)=\frac{3}{4}q(G)+\frac{3}{16}v_{G}^{01}.
\]

\end{lemma}

\begin{proof}
The set $\mathcal{F}_{vw}$ of all Boolean functions in two variables
$x_{v},x_{w}$ consists of 16 functions. For each solvable system $\mathcal{S}$
corresponding to $G$ we consider the system $\mathcal{S}\cup\{f=0\},$ where
$f\in\mathcal{F}_{vw}.$ Each system in $G+vw$ is thus constructed by adding
one such equation to a system in $G.$ If $\mathcal{S}$ is solvable system in
which $x_{v}=0$ in all solutions, then there are exactly $4$ of the $16$
polynomials $f\in\mathcal{F}_{vw},$ the extended system forces $x_{w}=0$ in
all solutions. The same is true when $x_{v}=1$ in all solutions. If
$\mathcal{S}$ allows both values for $x_{v}$ (i.e., is counted in $v_{G}^{01})
$, then for exactly $3$ out of $16$ functions $f$ the resulting system forces
$x_{w}=0$ in all solutions. Thus, in aggregate,
\[
w_{G+vw}^{0}=\frac{1}{4}v_{G}^{0}+\frac{1}{4}v_{G}^{1}+\frac{3}{16}v_{G}%
^{01}=\frac{1}{2}v_{G}^{0}+\frac{3}{16}v_{G}^{01}.
\]
Similarly,
\[
w_{G+vw}^{01}=\frac{1}{4}v_{G}^{0}+\frac{1}{4}v_{G}^{1}+\frac{9}{16}v_{G}%
^{01}=\frac{1}{2}v_{G}^{0}+\frac{9}{16}v_{G}^{01}.
\]
This proves part (a). As to part (c),
\begin{align*}
q(G+vw)  & =2w^{0}+w^{01}=^{\text{by (a)}}\frac{1}{16}(24v_{G}^{0}%
+15v_{G}^{01})=\\
\frac{1}{16}(12(2v_{G}^{0}+v_{G}^{01})+3v_{G}^{01})  & =^{\text{by Lemma
\ref{L1}}}\frac{3}{4}q(G)+\frac{3}{16}v_{G}^{01}.
\end{align*}
For part (b), exactly $9$ functions $f$ preserve the value of $x_{v}$ in the
extended system; that is, allow both $x_{v}=0$ and $x_{v}=1$ to appear in some
solutions of $\mathcal{S}\cup\{f=0\}$ if that was the case in $\mathcal{S}.$
Therefore, $v_{G+vw}^{01}=\frac{9}{16}v_{G}^{01}.$ Using the identity
\[
q(G+vw)=2v_{G+vw}^{0}+v_{G+vw}^{01}
\]
and solving for $v_{G+vw}^{0}$ yields
\begin{align*}
2v_{G+vw}^{0}  & =q(G+vw)-v_{G+vw}^{01}=^{\text{by (c)}}(\frac{3}{4}%
q(G)+\frac{3}{16}v_{G}^{01})-\frac{9}{16}v_{G}^{01}\\
& =\frac{3}{4}q(G)-\frac{3}{8}v_{G}^{01}.
\end{align*}
Hence,%
\[
v_{G+vw}^{0}=\frac{3}{8}q(G)-\frac{3}{16}v_{G}^{01}=\frac{3}{8}(2v_{G}%
^{0}+v_{G}^{01})-\frac{3}{16}v_{G}^{01}=\frac{3}{4}v_{G}^{0}+\frac{3}{16}%
v_{G}^{01}.
\]
This proves part (b). The proof is complete.

\end{proof}

For each $n\geq1,$ let $P_{n\text{ }}$be the path with the vertex set
$\{v_{0},...,v_{n}\},$ and the edge set $\{v_{i-1}v_{i},i=1,...,n\}.$ Also,
let $S_{n\text{ }}$be the star with the vertex set $\{w_{0},...,w_{n}\},$ and
the edge set $\{w_{i-1}w_{n},i=1,...,n\}.$ So, $w_{n}$ is the center of the
star $S_{n}.$

\begin{lemma}
\label{L3} Let $\{a_{n}\}_{1}^{\infty}$ and $\{b_{n}\}_{1}^{\infty}$ be
sequences defined recursively by
\begin{align*}
a_{1}  & =\frac{3}{16},\,b_{1}=\frac{9}{16},\text{ and for }n>1\\
a_{n}  & =\frac{1}{2}a_{n-1}+\frac{3}{16}b_{n-1},\\
b_{n}  & =\frac{1}{2}a_{n-1}+\frac{9}{16}b_{n-1}.
\end{align*}
Then $v_{n}^{0}=a_{n},v_{n}^{01}=b_{n},$ $q(P_{n})=2a_{n}+b_{n},$ and
$b_{n}\geq\left(  \frac{9}{16}\right)  ^{n}$. Further, $w_{n}^{01}=\left(
\frac{9}{16}\right)  ^{n},$ and $q(S_{n})=2\left(  \frac{3}{4}\right)
^{n}-\left(  \frac{3}{4}\right)  ^{2n}.$
\end{lemma}

\begin{proof}
\textbf{\ }We prove all statements by induction on $n.$ For $n=1,$ by Claim
\ref{Claim}, $v_{1}^{0}=\frac{3}{16},v_{1}^{01}=\frac{9}{16},$ i.e., by Lemma
\ref{L1}, $q(P_{1})=2v_{1}^{0}+v_{1}^{01}=2a_{1}+b_{1}.$ For $n>1,$ by Lemma
\ref{L4} (c),
\begin{align*}
q(P_{n})  & =\frac{3}{4}q(P_{n-1})+\frac{3}{16}b_{n-1}=\frac{3}{4}%
(2a_{n-1}+b_{n-1})+\frac{3}{16}b_{n-1}=\\
\frac{3}{2}a_{n-1}+\frac{15}{16}b_{n-1}  & =2(\frac{1}{2}a_{n-1}+\frac{3}%
{16}b_{n-1})+(\frac{1}{2}a_{n-1}+\frac{9}{16}b_{n-1})=2a_{n}+b_{n},
\end{align*}
and by part (a) $v_{n}^{0}=a_{n},v_{n}^{01}=b_{n}$.

For $n=1,b_{1}=\frac{9}{16},$ and by Lemma \ref{L4} (a) and induction
hypothesis $b_{n}=v_{P_{n}}^{01}=\frac{1}{2}v_{P_{n-1}}^{0}+\frac{9}%
{16}v_{P_{n-1}}^{01}>\frac{9}{16}b_{n-1}=\left(  \frac{9}{16}\right)  ^{n}.$

As for the star $S_{n},$ for $n=1$ we have $q(S_{1})=q(P_{1})=\frac{15}%
{16}=2(\frac{3}{4})^{1}-(\frac{3}{4})^{2},$ and, $w_{1}^{01}=\frac{9}{16}.$
For $n>1,$ by Lemma \ref{L4} (c),%
\begin{align*}
q(S_{n})  & =\frac{3}{4}q(S_{n-1})+\frac{3}{16}w_{n-1}^{01}\\
& =\frac{3}{4}\left(  2\left(  \frac{3}{4}\right)  ^{n-1}-\left(  \frac{3}%
{4}\right)  ^{2n-2}\right)  +\frac{3}{16}\left(  \frac{9}{16}\right)  ^{n-1}\\
& =\frac{3}{4}2\left(  \frac{3}{4}\right)  ^{n-1}-\frac{3}{4}\left(  \frac
{3}{4}\right)  ^{2n-2}+\frac{3}{16}\left(  \frac{3}{4}\right)  ^{2n-2}\\
& =2\left(  \frac{3}{4}\right)  ^{n}-\left(  \frac{3}{4}\right)  ^{2n}.
\end{align*}
Finally, by part (b), $w_{n}^{01}=\frac{9}{16}w_{n-1}^{01}=\left(  \frac
{9}{16}\right)  ^{n}.$
\end{proof}

\begin{lemma}
\label{L2}For each $s,t\geq1,$ it is
\[
b_{s}b_{t}<b_{s+t}.
\]

\end{lemma}

\begin{proof}
Let $s\geq1$ be fixed. We prove the statement by induction on $t.$ To be able
to carry out the induction step we prove a stronger statement%
\[
b_{s}b_{t}<b_{s+t}\quad\text{ and }\quad b_{s}a_{t}<a_{s+t}.
\]
For $t=1$ we get%
\begin{align*}
b_{s}b_{1}  & =b_{s}\frac{9}{16}<\frac{1}{2}a_{s}+\frac{9}{16}b_{s}%
=b_{s+1},\text{ }\\
b_{s}a_{1}  & =b_{s}\frac{3}{16}<\frac{1}{2}a_{s}+\frac{3}{16}b_{s}=a_{s+1}%
\end{align*}
as $a_{s}>0$ for all $s\geq1.$ Assume that $b_{s}b_{k}<b_{s+k}$ and
$b_{s}a_{k}<a_{s+k}$ has been proved for all $k<t$. Then%
\begin{align*}
b_{s}b_{t}  & =b_{s}(\frac{1}{2}a_{t-1}+\frac{9}{16}b_{t-1})<\frac{1}%
{2}a_{s+t-1}+\frac{9}{16}b_{s+t-1}=b_{s+t},\\
b_{s}a_{t}  & =b_{s}(\frac{1}{2}a_{t-1}+\frac{3}{16}b_{t-1})<\frac{1}%
{2}a_{s+t-1}+\frac{3}{16}b_{s+t-1}=a_{s+t}.
\end{align*}

\end{proof}

\subsection{Trees with extremal consistency probabilities}

The next theorem characterizes the trees that achieve extremal consistency probabilities.

\begin{theorem}
\bigskip\label{trees}For every tree $T_{n}$ with $n\geq3$ edges that is
neither a path or a star, it is
\[
q(S_{n})<q(T_{n})<q(P_{n}).
\]

\end{theorem}

\begin{proof}
We prove the theorem by induction on the number of edges. For the induction
step it is convenient to prove a stronger statement: For every tree $T_{n}$
with $n\geq3$ edges that is neither a path or a star, it is $q(S_{n}%
)<q(T_{n})<q(P_{n})$, and furthermore, for every vertex $v$ of each tree
(including a star and a path) with $n\geq1$ edges it holds
\begin{equation}
\left(  \frac{9}{16}\right)  ^{n}\leq v^{01}\leq b_{n}.\label{2}%
\end{equation}
Preliminaries and values from the lemmata. Let $P_{n\text{ }}$be the path with
the vertex set $\{v_{0},...,v_{n}\},$ and the edge set $\{v_{i-1}%
v_{i},i=1,...,n\},$ and $S_{n}$ be a star with the center $w_{n}.$ We recall
that by Lemma \ref{L3},
\begin{align*}
q(S_{n})  & =2\left(  \frac{3}{4}\right)  ^{n}-\left(  \frac{3}{4}\right)
^{2n},w_{n}^{01}=(\frac{9}{16})^{n},\\
q(P_{n})  & =2a_{n}+b_{n},\text{ and }b_{n}\geq\left(  \frac{9}{16}\right)
^{n}.
\end{align*}
From the recurrence (given in Lemma \ref{L3}) we obtain the concrete values
$a_{2}=\frac{51}{16^{2}},b_{2}=\frac{105}{16^{2}},a_{3}=\frac{723}{16^{3}},$
$b_{3}=\frac{1353}{16^{3}}.$

\textit{Base case} $n=3.$ When $n=3$ there is no tree which is neither a path
nor a star, so the inequality in the statement reduces to showing
$\ q(S_{3})<q(P_{3}).$

{}Using the numbers above%
\[
q(S_{3})=\frac{2727}{16^{3}}<2\frac{723}{16^{3}}+\frac{1353}{16^{3}}%
=\frac{2799}{16^{3}}=q(P_{3})
\]
as required.

As to (\ref{2}), it needs to be proved for all trees $T$ with $n=1,2,$ and $3
$ edges. If $n=1,$ then $T=P_{1},$ and $v_{0}^{01}=v_{1}^{01}=b_{1}=\frac
{9}{16}.$ For $n=2,T=P_{2},$ and $v_{0}^{01}=v_{2}^{01}=b_{2}\geq(\frac{9}%
{16})^{2}\,.$ The vertex $v_{1}$ can be seen as the center $w_{2}$ of the star
$S_{2}$ and thus $w_{2}^{01}=(\frac{9}{16})^{2}.$ For $n=3,$ \thinspace
$T=S_{3}$ or $T=P_{3}.$ By symmetry it suffices to check vertices $v_{2}$ and
$v_{3}$ of $P_{3},$ and the center and a pendant vertex of $S_{3}.$

For $v_{3}$ it holds
\[
v_{3}^{01}=b_{3}\geq\left(  \frac{9}{16}\right)  ^{3}
\]
By Lemma \ref{L4} we have $v_{2\text{ }}^{01}=^{(b)}\frac{9}{16}b_{2}%
<^{(a)}\frac{1}{2}a_{2}+\frac{9}{16}b_{2}=b_{3}.$ Also $\frac{9}{16}%
b_{2}=\frac{9}{16}\frac{105}{16^{2}}>\left(  \frac{9}{16}\right)  ^{3},$ thus
$v_{2}^{01}$ satisfies (\ref{2}) \ as well. For $w_{3},$ the center of
$S_{3},$ By Lemma \ref{L3}
\[
w_{3}^{01}=\left(  \frac{9}{16}\right)  ^{3}<\frac{1353}{16^{3}}=b_{3}.
\]
Finally, (\ref{2}) will be shown for a pending vertex $w$ of $S_{3}:$ view
$S_{3}$ as $P_{2}+v_{1}w,$ where $v_{1}$ is the middle vertex of $P_{2}.$ By
Lemma \ref{L4} in $P_{2}$ it is $v_{1}^{0}=\frac{63}{16^{2}},v_{1}^{01}%
=\frac{81}{16^{2}}.$ Then $w_{P_{2}+v_{1}w}^{01}=\frac{1}{2}v_{1}^{0}+\frac
{9}{16}v_{1}^{01}=\left(  \frac{9}{16}\right)  ^{3}+\frac{8\cdot63}{16^{3}%
}=\frac{1233}{16^{3}}<\frac{1353}{16^{3}}=b_{3}.$ So (\ref{2}) holds for all
vertices when $n=3.$

\textit{Inductive step}. Let now $T$ be a tree with $n>3$ edges different from
a path and a star, and $vw$ be a pending edge with $\deg_{T}(w)=1$ such that
$T-vw$ is not a path. By induction hypothesis it holds $q(T-vw)<q(P_{n-1}),$
$v_{T-vw}^{01}\leq b_{n-1},$ and we get%
\begin{align*}
q(T)  & =q((T-vw)+vw)=^{\text{Lemma \ref{L4}(c)}}\frac{3}{4}q(T-vw)+\frac
{3}{16}v_{T-vw}^{01}\\
& <\frac{3}{4}q(P_{n-1})+\frac{3}{16}b_{n-1}=\frac{3}{4}(2a_{n-1}%
+b_{n-1})+\frac{3}{16}b_{n-1}\\
& =2(\frac{1}{2}a_{n-1}+\frac{3}{16}b_{n-1})+(\frac{1}{2}a_{n-1}+\frac{9}%
{16}b_{n-1})=2a_{n}+b_{n}=q(P_{n}).
\end{align*}
As for the left inequality, there is a pending edge $vw$ in $T$, $\deg
_{T}(w)=1,$ such that $T-vw$ is not a star. By induction hypothesis
$q(T-vw)>q(S_{n-1}),$ and $v_{T-vw}^{01}\geq\left(  \frac{9}{16}\right)
^{n-1}.$ Combining it with Lemma \ref{L4} we get
\begin{align*}
q(T)  & =\frac{3}{4}q(T-vw)+\frac{3}{16}v_{T-vw}^{01}\\
& >\frac{3}{4}\left(  2\left(  \frac{3}{4}\right)  ^{n-1}-\left(  \frac{9}%
{16}\right)  ^{n-1}\right)  +\frac{3}{16}\left(  \frac{9}{16}\right)  ^{n-1}\\
& =2\left(  \frac{3}{4}\right)  ^{n}-\left(  \frac{9}{16}\right)  ^{n}%
=q(S_{n}).
\end{align*}
Now we will prove (\ref{2}). First, let $\deg_{T}(w)=1,$ and $v$ be a vertex
of $T$ adjacent to $w$. Assume by contradiction that $b_{n}<w^{01}.$ Then
$2w^{0}+w^{01}=q(T)<q(P_{n})=2a_{n}+b_{n}$ implies $w^{0}<a_{n}.$ In turn we
get%
\begin{align*}
b_{n}  & <w^{01}\Longrightarrow\frac{1}{2}a_{n-1}+\frac{9}{16}b_{n-1}<\frac
{1}{2}v_{T-vw}^{0}+\frac{9}{16}v_{T-vw}^{01},\text{ and}\\
w^{0}  & <a_{n}\Longrightarrow\frac{1}{2}v_{T-vw}^{0}+\frac{3}{16}%
v_{T-vw}^{01}<\frac{1}{2}a_{n-1}+\frac{3}{16}b_{n-1}.
\end{align*}
Adding and rearranging the above inequalities leads to $b_{n-1}<v_{T-vw}%
^{01},$ which contradicts induction hypothesis.

On the other hand,
\[
w^{01}=\frac{1}{2}v_{T-vw}^{0}+\frac{9}{16}v_{T-vw}^{01}>\frac{9}{16}%
v_{T-vw}^{01}\geq\frac{9}{16}\left(  \frac{9}{16}\right)  ^{n-1}=\left(
\frac{9}{16}\right)  ^{n}
\]
by induction hypothesis. Now we consider a vertex $w$ with $\deg_{T}(w)>1.$ If
$T$ is a star, then $w$ is its center, and $w^{01}=(\frac{9}{16})^{n}.$ In the
other case, there is a vertex $z$, with $\deg_{T}(z)>1,$ adjacent to $w$. Let
$T_{w}$ and $T_{z}$ be the subtree of $T-wz$ containing the vertex $w$ and
$z,$ respectively. By similar reasoning as in Lemma \ref{L4} we have%
\[
w_{T}^{01}=w_{T_{w}}^{01}z_{T_{z}}^{01}\frac{9}{16}+w_{T_{w}}^{01}z_{T_{z}%
}^{0}\frac{1}{2}=w_{T_{w}}^{01}\left(  \frac{1}{2}z_{T_{z}}^{0}+\frac{9}%
{16}z_{T_{z}}^{01}\right)  =w_{T_{w}}^{01}y_{T_{z}+yz}^{01},
\]
where $y$ is a pending new vertex vertex in $T_{z}+yz.$ By induction
hypothesis, $\left(  \frac{9}{16}\right)  ^{s}\leq w_{T_{w}}^{01}\leq b_{s}$
and $\left(  \frac{9}{16}\right)  ^{t+1}\leq y_{T_{z}+yz}^{01}\leq b_{t+1},$
where $s$ and $t$ is the number of edges in $T_{w}$ and $T_{z},$ respectively.
Clearly, $s+t=n-1.$ Thus $w_{T}^{01}=w_{T_{w}}^{01}y_{T_{z}+yz}^{01}%
\geq\left(  \frac{9}{16}\right)  ^{n},$ while, by Lemma \ref{L2}, we get
$w_{T}^{01}\leq b_{s}b_{t+1}<b_{n}$. The proof is complete.
\end{proof}

\subsection{Probability of a path}

In this subsection, we derive an explicit formula for $q(P_{n}).$

\bigskip The roots of the quadratic equation $x^{2}=\frac{17}{16}x-\frac
{3}{16}$will be denoted by $s=\frac{17+\sqrt{97}}{32},$ and $t=\frac
{17-\sqrt{97}}{32}.$

\begin{lemma}
\label{L5} Let $\{d_{n}\}_{1}^{\infty}$ be a sequence defined recursively by
$d_{n}=\frac{17}{16}d_{n-1}-\frac{3}{16}d_{n-2},$ for $n\geq3.$ Then the
particular solution satisfying $d_{1}=a,d_{2}=b$ is given by%
\[
d_{n}=\frac{b-at}{s(s-t)}s^{n}+\frac{-b+as}{t(s-t)}t^{n}.
\]

\end{lemma}

\begin{proof}
The general solution of the recurrent relation is
\[
d_{n}=c_{1}s^{n}+c_{2}t^{n},
\]
where $s$ and $t$ are the roots of the characteristic equation. The constants
$c_{1}$ and $c_{2}$ satisfy the system%
\[
a=c_{1}s^{1}+c_{2}t^{1},\text{ }b=c_{1}s^{2}+c_{2}t^{2}
\]
Using standard methods yields
\[
c_{1}=\frac{b-a\cdot t}{s(s-t)}\text{ and }c_{2}=\frac{-b+a\cdot s}{t(s-t)}
\]
The result follows.
\end{proof}

By Lemma \ref{L3}, $q(P_{n})=2a_{n}+b_{n}.$ A direct computation gives
$q(P_{1})=\frac{15}{16},$ and $q(P_{2})=\frac{207}{256}.$ Now we state the
main result of this subsection$:$

\begin{theorem}
\label{Tpath} \ \bigskip Let $c=\frac{q(P_{2})-q(P_{1})\cdot t}{s(s-t)}.$
Then, for $n\geq1,$%
\[
q(P_{n})=cs^{n}+(1-c)t^{n}\approx1.16\cdot s^{n}-0.16\cdot t^{n},
\]

\end{theorem}

where $s=\frac{17+\sqrt{97}}{32},t=\frac{17-\sqrt{97}}{32}.$

\begin{proof}
$\ $Let $\{a_{n}\}_{1}^{\infty}$ and $\{b_{n}\}_{1}^{\infty}$ be sequences
defined in Lemma \ref{L3}.

Now a second order recurrence relation for $\{b_{n}\}_{1}^{\infty}$ will be
derived.
\begin{align*}
\frac{1}{2}a_{n-1}  & =b_{n}-\frac{9}{16}b_{n-1},\\
a_{n}  & =b_{n}-\frac{9}{16}b_{n-1}+\frac{3}{16}b_{n-1}=b_{n}-\frac{6}%
{16}b_{n-1},\\
a_{n-1}  & =b_{n-1}-\frac{6}{16}b_{n-2},\\
b_{n}  & =\frac{1}{2}(b_{n-1}-\frac{6}{16}b_{n-2})+\frac{9}{16}b_{n-1}%
=\frac{17}{16}b_{n-1}-\frac{3}{16}b_{n-2}.\text{ }%
\end{align*}
Thus $\{b_{n}\}$ satisfies the same recurrence relation as in Lemma \ref{L5},
so
\[
b_{n}=\frac{b_{2}-b_{1}\cdot t}{s(s-t)}s^{n}+\frac{-b_{2}+b_{1}\cdot
s}{t(s-t)}t^{n}
\]
Similarly, one can check that $\{a_{n}\}_{1}^{\infty}$ follows the same
recurrence relation%
\[
a_{n}=\frac{17}{16}a_{n-1}-\frac{3}{16}a_{n-2}
\]
so, again by Lemma \ref{L5},
\[
a_{n}=\frac{a_{2}-a_{1}\cdot t}{s(s-t)}s^{n}+\frac{-a_{2}+a_{1}\cdot
s}{t(s-t)}t^{n}.
\]
By Lemma \ref{L3}, $q(P_{n})=2a_{n}+b_{n},$ and substituting expression for
$a_{n}$ and $b_{n},$ we get
\begin{align*}
q(P_{n})  & =2a_{n}+b_{n}\\
& =\frac{(2a_{2}+b_{2})-(2a_{1}+b_{1})\cdot t}{s(s-t)}s^{n}+\frac
{-(2a_{2}+b_{2})+(2a_{1}+b_{1})\cdot s}{t(s-t)}t^{n}\\
& =\frac{q(P_{2})-q(P_{1})\cdot t}{s(s-t)}s^{n}+\frac{-q(P_{2})+q(P_{1})\cdot
s}{t(s-t)}t^{n}.
\end{align*}
To finish the proof it needs to be shown that
\[
\frac{q(P_{2})-q(P_{1})\cdot t}{s(s-t)}+\frac{-q(P_{2})+q(P_{1})\cdot
s}{t(s-t)}=1.
\]
A simple manipulation and Vieta's formulae $s+t=\frac{17}{16},st=\frac{3}{16}
$ yields%
\begin{align*}
\frac{q(P_{2})-q(P_{1})\cdot t}{s(s-t)}+\frac{-q(P_{2})+q(P_{1})\cdot
s}{t(s-t)}  & =\frac{-q(P_{2})+q(P_{1})(s+t)}{st}=\\
\frac{-\frac{207}{256}+\frac{15}{16}\frac{17}{16}}{\frac{3}{16}}  &
=\frac{15\cdot17-207}{48}=1.
\end{align*}
Therefore, $q(P_{n})=cs^{n}+(1-c)t^{n}.$ This concludes the proof.
\end{proof}

\subsection{\bigskip Cycle-free graph with maximum probability}

In this section, we consider cycle-free graphs (i.e., forests) without
isolated vertices, since isolated vertices would correspond to variables that
are not active in any equation. Clearly, we have $n>m$; that is, the number of
vertices exceeds the number of edges. We characterize the cycle-free graph
that maximizes the consistency probability. The next two lemmas provide key
ingredients for proving the uniqueness of the extremal graph.

\begin{lemma}
\label{L7} Let $C_{1},...,C_{t}$ be connectivity components of a graph $G.$
Then
\[
q(G)=\prod\limits_{i=1}^{k}q(C_{i}).
\]

\end{lemma}

\begin{proof}
Components $C_{1},...,C_{t}$ correspond to systems of equations on disjoint
sets of variables. Therefore their solutions are independent of each other.
The result follows.
\end{proof}

\begin{lemma}
\label{L6}Let $k\geq1,$ and $d\geq k+2.$ Then\bigskip%
\[
q(P_{k+1})\cdot q(P_{d-1})>q(P_{d})\cdot q(P_{k}).
\]

\end{lemma}

\begin{proof}
As $q(P_{n})=cs^{n}+(1-c)t^{n}$, we need to prove
\[
\left(  cs^{k+1}+(1-c)t^{k+1}\right)  (cs^{d-1}+(1-c)t^{d-1})>(cs^{d}%
+(1-c)t^{d})(cs^{k}+(1-c)t^{k}).
\]
Multiplying out factors and subtracting matching terms yields
\[
c(1-c)(s^{d-1}t^{k+1}+s^{k+1}t^{d-1})>c(1-c)(s^{d}t^{k}+s^{k}t^{d}).
\]
As $c>1,$ it is $c(1-c)<0,$ hence
\begin{align*}
s^{d-1}t^{k+1}+s^{k+1}t^{d-1}  & <s^{d}t^{k}+s^{k}t^{d},\\
s^{k}t^{d-1}(s-t)  & <s^{d-1}t^{k}(s-t),\\
t^{d-1-k}  & <s^{d-1-k}.
\end{align*}
The inequality holds as $s>t,$ and $d>k+1.$
\end{proof}

\begin{theorem}
\label{Main}Let G be a cycle-free graph on $n$ vertices and $m<n$ edges
without isolated vertices. Then $\max q(G)$ is attained by a graph whose all
components $T_{i}$ are paths, and the number of vertices $n(T_{i})$ of those
components are as equal as possible; i.e.,%
\[
\left\vert n(T_{i})-n(T_{j})\right\vert \leq1\text{ for all }i,j.
\]

\end{theorem}

\begin{proof}
Let $G$ be a a cycle-free graph with the maximum consistency probability. As
$G$ is a forest with no isolated vertices, each component is a tree with at
least one edge. By Theorem \ref{trees}, each component $C_{i}$ is a paths, and
by Lemma \ref{L7}, $q(G)=\prod\limits_{i}^{{}}q(C_{i}).$ If $G$ contained two
components $P_{k}$ and $P_{d}$, where $d\geq k+2,$ a graph resulting from
replacing these two paths by $P_{d-1}$ and $P_{k+1}$ would have greater
consistency probability, see Lemma \ref{L6}. Therefore, in $G$, the orders of
any two components of must differ by at most $1.$
\end{proof}

\subsection{Consistency probability of a cycle}

As usual, the cycle with $n$ edges is denoted by $C_{n}.$ We propose the
following conjecture

\begin{conjecture}
The consistency probability of the cycle $C_{n}$ is given by%
\[
q(C_{n})=s^{n}+t^{n}-\frac{1}{4^{n}}-\frac{1}{8^{n-1}},
\]
where $s=\frac{17+\sqrt{97}}{32},t=\frac{17-\sqrt{97}}{32}.$
\end{conjecture}

As a supporting evidence we have used a computer to calculate the values of
$q(C_{n})$ for $n=3,...,7,$ and in each case the computed values agree with
those given by the formula.\bigskip

\noindent\textbf{Conflict of Interest\medskip}

The manuscript is original, has not been published previously, and is not
under consideration for publication elsewhere. All authors have approved the
submission. None of the authors have a conflict of interest to disclose.


\begin{thebibliography}{99}                                                                                               %
\bibitem {Ach00}D. Achlioptas. 2000. "Setting 2 variables at a time yields a
new lower bound for random 3-SAT (extended abstract)". \textit{Proc. 32nd ACM
Symposium on Theory of Computing}: 28--37.

\bibitem {Bard}G. V. Bard. 2009. \textit{Algebraic Cryptanalysis}. Springer.

\bibitem {CR}V. Chv\'{a}tal and B. Reed.\emph{\ 1992. "}Mick gets some (the
odds are on his side)". \emph{In 33th Annual Symposium on Foundations of
Computer Science}: 620--627. IEEE Comput. Soc. Press.

\bibitem {DSS}J. Ding, A. Sly, and N. Sun. 2015.\emph{\ "}Proof of the
Satisfiability Conjecture for Large k", \emph{In Proceedings of the 47-th
Annual ACM Symposium on Theory of Computing STOC 2015}: 59--68.

\bibitem {MiniSat}N. E\'{e}n, N. S\"{o}rensson. \emph{MiniSat home page}, http://minisat.se.

\bibitem {Goerdt92}A. Goerdt. 1992. "A threshold for unsatisfiability".
\emph{Mathematical Foundations of Computer Science}, 17th Intl. Symposium,
I.M. Havel and V. Koubek, Eds., Lecture Notes in Computer Science 629:
264--274, Springer.

\bibitem {JSV00}S. Janson, Y.C. Stamatiou, and M. Vamvakari. 2000. "Bounding
the unsatisfiability threshold of random 3-SAT". \emph{Rand. Struc. Alg}. 17:103-116.

\bibitem {Odlyzko}B.A LaMacchia, and A.M Odlyzko. 1991. "Solving large sparse
linear systems over finite fields". In A. J. Menezes \& S. A. Vanstone (Eds.),
\emph{Advances in Cryptology -- CRYPTO '90 }(Lecture Notes in Computer
Science, 537: 109--133). Springer, Heidelberg.

\bibitem {iS09}I. Semaev. 2009. "Sparse algebraic equations over finite
fields". \emph{SIAM J. Comput., 39(2)}: 388--409.

\bibitem {iS2013}I. Semaev. 2013. "Improved agreeing-gluing algorithm".
\emph{Math. in Comp. Science} 7(3): 321--339.

\bibitem {GS}G. B. Sorkin. 2001. "Some Notes on Random Satisfiability". In
SAGA 2001. \emph{Lecture Notes in Computer Science}, 2264: 117-130. Springer.
\end{thebibliography}
\end{document}